\documentclass[letterpaper]{amsart}
\usepackage[utf8]{inputenc}

\usepackage{amsmath,xparse}
\usepackage{amsfonts}
\usepackage{amssymb}
\usepackage{enumerate}
\usepackage{enumitem}
\usepackage{graphicx}
\usepackage{microtype}
\usepackage{lipsum}
\usepackage{comment}
\usepackage[english]{babel}
\usepackage[hang]{footmisc}
\usepackage{geometry}
 \usepackage{hyperref}
\hypersetup{
    colorlinks=true,
    linkcolor=blue,
    citecolor=blue,
    urlcolor=blue
}

\newcommand{\footlabel}[2]{%
    \refstepcounter{footnote}\label{#1}%
    \footnotetext[\thefootnote]{#2}%
    $^{\ref{#1}}$%
}

\usepackage{etoolbox}
\makeatletter
\patchcmd{\@sect}{\@addtoreset{footnote}{#1}}{}{}{}
\makeatother

\let\latexchi\chi
\makeatletter
\renewcommand\chi{\@ifnextchar_\sub@chi\latexchi}
\newcommand{\sub@chi}[2]{
  \@ifnextchar^{\subsup@chi{#2}}{\latexchi_{\raisebox{-.7ex}{\ensuremath{\scriptstyle #2}}}}%
}
\newcommand{\subsup@chi}[3]{
  \latexchi_{\raisebox{-1ex}{\ensuremath{\scriptstyle #1}}}^{#3}%
}
\makeatother

\theoremstyle{definition}
\newtheorem{thm}{Theorem}[section]
\newtheorem{prop}[thm]{Proposition}
\newtheorem{lem}[thm]{Lemma}
\newtheorem{cor}[thm]{Corollary}

\newtheorem{defn}[thm]{Definition}
\newtheorem{ex}[thm]{Example}

\theoremstyle{remark}
\newtheorem{rmk}{Remark}[section]

\title{Convex Geometry of Building Sets}
\author{Spencer Backman}
\email{sbackman@uvm.edu}
\author{Richard Danner}
\email{rsdanner@uvm.edu}

\begin{document} 
\begin{abstract}
Building sets were introduced in the study of wonderful compactifications of hyperplane arrangement complements and were later generalized to finite meet-semilattices.   Convex geometries, the duals of antimatroids, offer a robust combinatorial abstraction of convexity.  Supersolvable convex geometries and antimatroids appear in the study of poset closure operators, Coxeter groups, and matroid activities.  
We prove that the building sets on a finite meet-semilattice form a supersolvable convex geometry.  As an application, we demonstrate that building sets and nested set complexes respect certain restrictions of finite meet-semilattices, unifying and extending results of several authors.

\end{abstract}

\maketitle
\thispagestyle{empty}
\section{Introduction}

De Concini and Procesi's celebrated wonderful compactifications of a hyperplane arrangement complement proceed by blow ups of certain flats of the arrangement \cite{de1995wonderful}.  There are choices for which flats are blown up, and each set of choices is encoded by a \emph{building set}.  Feichtner and Kozlov \cite{feichtner2004incidence} extended the notion of a building set from the lattice of flats of a hyperplane arrangement to arbitrary finite meet-semilattices\footnote{Henceforth, all posets in this paper will be assumed to be finite.}, and observed that this generalization admits connections to other topics such as stellar subdivisions of fans.  

The Bergman fan of a matroid was introduced by Ardila-Klivans \cite{ardila2006bergman} as a matroidal generalization of the tropicalization of a linear space.  Feichtner-Sturmfels \cite{feichtner2005matroid} showed that the nested set complex associated to a building set for the lattice of flats gives a unimodular triangulation of the Bergman fan.  In seminal work of Adiprasito-Huh-Katz \cite{adiprasito2018hodge}, the Chow ring of a matroid, introduced earlier by Feichtner-Yuzvinsky \cite{feichtner2004chow}, was utilized for settling the Heron-Rota-Welsh conjecture on the log-concavity of the coefficients of the characteristic polynomial of a matroid.  Although Feichtner-Yuzvinsky's construction is defined for arbitrary building sets, Adiprasito-Huh-Katz found that the maximum building set was sufficient for their goals.\footnote{While the Bergman fans of \cite{adiprasito2018hodge} are not Bergman fans induced by building sets on the lattice of flats, we note a connection in footnote  \ref{oderfilt}.} Much of the subsequent work by researchers studying Chow rings of matroids has been focused on the maximum building set with a few important notable exceptions: 
the conormal fan of a matroid\footnote{The biflats, viewed as a subset of $\mathcal{L}(M)^{op}\times \mathcal{L}(M^{\perp})$, are closed under taking joins and thus satisfy condition (\ref{bscon2}) of Proposition \ref{usefulchar}.  While the biflats do not satisfy condition (\ref{bscon1}) of Proposition \ref{usefulchar}, the image of the biflats in $\Sigma_M \times \Sigma_{M^{\perp}}$ contains the image of the irreducibles.  As the biflats are closed under taking joins, the would-be nested set complex associated to the biflats is the order complex, of which the conormal fan is a distinguished codimension-2 subcomplex.} \cite{ardila2023lagrangianC, ardila2023lagrangian,nathanson2023topology},
the augmented Bergman fan of a matroid \cite{braden2020singular,braden2022semi, eur2023stellahedral,liao2022stembridge, mastroeni2023chow}, the Bergman fan of a polymatroid \cite{crowley2022bergman, eur2023intersection,pagaria2023hodge}\footnote{It follows from each of the works \cite{amini2020hodge, ardila2023lagrangianC, pagaria2023hodge} that the K\"ahler package holds for the Chow ring of a matroid with respect to an arbitrary building set.}, and a few additional works \cite{ 
coron2022matroids,coron2025supersolvability,eur2025building}.

Perhaps surprisingly, there has been no work exploring the systematic construction of  building sets for the lattice of flats of a matroid -- this apparent gap in the literature was the genesis of this paper. Although our primary motivation is the study of matroids, we have found that our results hold for arbitrary meet-semilattices.  
In their original article, de Concini-Procesi provided three different equivalent descriptions of a building set \cite{de1995wonderful}.  Their most elegant characterization was popularized by Postnikov \cite{postnikov2009permutohedra} and Feichtner-Sturmfels \cite{feichtner2005matroid} for the Boolean lattice, but has largely been overlooked by researchers studying matroids.  We begin by extending this characterization to a general meet-semilattice $\mathcal{L}$, e.g. a geometric lattice.  
We apply this result to demonstrate that the building sets for $\mathcal{L}$ form an intersection lattice.  It follows that given an arbitrary subset $X \subseteq \mathcal{L}$, there exists a smallest building set $B$ which contains $X$, thus the building sets for $\mathcal{L}$ determine a closure operator; this generalizes a result of Feichtner-Sturmfels \cite{feichtner2005matroid} from the Boolean lattice to general meet-semilattices.   

Convex geometries are a special class of closure systems introduced independently by Edelman  and Jamison \cite{edelman1980meet,edelman1985theory,jamison1982perspective} providing a combinatorial abstraction of convex hulls in finite point configurations and upper ideals in posets. A fundamental result of Bj\"orner \cite{bjorner1983matroids} states that convex geometries are dual to antimatroids, which are an important class of greedoids. 
Supersolvable antimatroids and convex geometries were introduced implicitly in the work of Stanley \cite{stanley1972supersolvable}, and were developed in earnest by Armstrong \cite{armstrong2009sorting}.  They arise naturally in the study of poset closure operators \cite{hawrylycz1993lattice}, Coxeter groups \cite{armstrong2009sorting}, and matroid activities \cite{gillespie2020convexity}.  We prove that the building sets on a meet-semilattice form a supersolvable convex geometry.

We demonstrate the utility of our results  by proving that the building sets and nested set complexes behave well with respect to restrictions of meet-semilattices.  This result unifies the work of Mantovani-Padrol-Pilaud on oriented matroids \cite{mantovani2025facial} with Bergman fan constructions arising in the study of Hodge theory for matroids due to Braden-Huh-Matherne-Proudfoot-Wang \cite{braden2022semi},  Crowley-Huh-Larson-Simpson-Wang \cite{crowley2022bergman}, and Eur-Larson \cite{eur2023intersection}.  

\section{Semilattices and Building Sets}

In this section we review the theory of posets, semilattices, and building sets.  We prove Proposition \ref{usefulchar} which extends a characterization of building sets due to de Concini and Procesi  from intersection lattices of subspace arrangements to general meet-semilattices.  The section concludes with Proposition \ref{lattice} which states that the collection of building sets on a meet-semilattice ordered by inclusion forms an intersection lattice, generalizing an observation of Feichtner-Sturmfels from the Boolean lattice to general meet-semilattices.

\begin{defn}
A  \emph{partially ordered set}, or \emph{poset}, is an ordered pair $(\mathcal{P}, \leq)$ where $\mathcal{P}$ is a set and $\leq$ is a relation such that for all $x,y,z \in \mathcal{P}$ 

\begin{enumerate}
    \item $x \leq x$ 
    
\item $x \leq y$ and $y \leq x$ implies $x=y$ 

\item  $x\leq y$ and $y \leq z$ imply $x \leq z$ 
\end{enumerate}
\end{defn}

We write $x<y$ if $x\leq y$ and $x\neq y$, and we define $x>y$ similarly.  We say that $y$ \emph{covers} $x$ if $x <y$ and there is no element $z \in \mathcal{P}$ such that $x<z<y$.  We define the subset $\mathcal{P}_{\leq x}$ to be $\{z \in \mathcal{P}:z \leq x\}$.  
For a subset $B$ of $\mathcal{P}$ with the induced order, and $x \in \mathcal{P}$, we define $B_{\leq x}=\{z \in B: z \leq x\}$, and $\text{max}\,B_{\leq x}$ as the set of maximal elements in $B_{\leq x}$ with respect to $\leq$. An interval $[x,y] \subseteq \mathcal{P}$ is the set of elements $z\in \mathcal{P}$ such that $x \leq z \leq y$. 
A linear extension $<_L$ of a partially ordered set $(\mathcal{P},\leq)$ is a total order of $\mathcal{P}$ such that if $x,y \in \mathcal{P}$ with $x<y$, then $x <_L y$.  Let $(\mathcal{P}_1,\leq_1)$ and $(\mathcal{P}_2, \leq_2)$ be posets. A function $f:\mathcal{P}_1 \rightarrow \mathcal{P}_2$ is an \emph{order embedding} if $x\leq_1 y$ precisely when $f(x) \leq_2 f(y)$.  Note that an order embedding is necessarily injective.  If $f$ is an inclusion map, we may write $ \mathcal{P}_1 \hookrightarrow \mathcal{P}_2$.  An order embedding is an \emph{isomorphism} if it is bijective.   Let $\{(\mathcal{P}_k,\leq_k):1\leq k \leq n\}$ be a collection of posets.

The product $\prod_{k=1}^n\mathcal{P}_k$ is naturally equipped with a partial order $\leq$ where for $x,y \in \prod_{k=1}^n\mathcal{P}_k$ we have $x\leq y$ if and only if $x_k \leq_k y_k$ for all $k$ with $1\leq k\leq n$.  If a poset has a minimum element, we may denote this element by $\hat{0}$, and if a poset has a maximum element, we may denote this element by  $\hat{1}$.  An element $a \in \mathcal{P}$ is an \emph{atom} if it covers $\hat{0}$.  A poset is \emph{irreducible} if it is not a product of two other posets, each
consisting of at least two elements. Let $I(\mathcal{P})=\{x \in \mathcal{P} \mid [\hat{0},x]~\text{is irreducible}\}$ be the \emph{set of irreducibles of $\mathcal{P}$}.  The set $\text{max}\,I(\mathcal{P})_{\leq x}$ is the set 
 of \emph{elementary divisors}  of $x$. 

\begin{lem}\label{posetfactorlem}\cite[Proposition 2.1]{feichtner2004incidence}
Let $(\mathcal{P},\leq)$ be a partially ordered set with a minimum element $\hat{0}$ and a maximum element $x$.  Let  $\text{max}\,I(\mathcal{P})=\{y_1,...,y_l\}$, then

$$\mathcal{P}\cong \prod_{i=1}^l[\hat{0},y_i].$$
  
\end{lem}

\begin{proof}
    If $x$ is irreducible, then the statement is trivial.  Otherwise $x$ is reducible and, by induction, there exists some factorization $[\hat{0},x]=\mathcal{P} \cong \prod_{i=1}^k\mathcal{P}_i$ with $k\geq 2$, such that $\mathcal{P}_i$ is irreducible and $|\mathcal{P}_i|\geq 2$ for each $i$.   We claim that, after some potential reordering, $\mathcal{P}_i \cong [\hat{0},y_i]$ and $k=l$.  First we observe that each $\mathcal{P}_i$ must have a minimum element, which we denote $\hat{0}$, and a maximum element $z_i$, otherwise $\mathcal{P}$ would have more than one minimal or maximal element.  Let $\mathcal{P}_i = [\hat{0},z_i]$ and let $\phi:\prod_{i=1}^k\mathcal{P}_i \rightarrow \mathcal{P}$ be an explicit isomorphism.  
 For any fixed $y_j$, there exists $w_i \in \mathcal{P}_i$ for $1\leq i\leq k$ such that $\phi(\prod_{i=1}^kw_i)=y_j$.  It follows from the definition of a poset isomorphism that $[\hat{0}, y_j] \cong \prod_{i=1}^k[\hat{0},w_i].$ Therefore, there is a unique index $i$ for which $w_i \neq \hat{0}$.  By assumption, each $z_i$ is irreducible.  Letting $\phi(z_i)$ denote $\phi(\hat{0},\ldots, z_i, \ldots, \hat{0})$, we have $y_j \leq \phi(z_i)$.  By the definition of $y_j$, it must be that $y_j = \phi(z_i)$.  Thus the map $\phi$ induces a canonical bijection between $\{z_1, \ldots, z_k\}$ and $\{y_1, \ldots, y_l\}$, and the statement of the claim follows. 
\end{proof}

We note that the factorization in the statement of Lemma \ref{posetfactorlem} is the unique finest factorization of $\mathcal{P}$.

\begin{defn}
Let $\mathcal{P}$ be a poset and $x,y \in \mathcal{P}$.  The greatest lower bound of elements $x$ and $y$, if it exists, is the \emph{meet} of $x$ and $y$, denoted $x\wedge y$.
The least upper bound of $x$ and $y$, if it exists, is the \emph{join} of $x$ and $y$, denoted $x\vee y$.  A poset $\mathcal{L}$ is called a \emph{meet-semilattice} if the meet of any two of its elements exists in $\mathcal{L}$.  We define a \emph{join-semilattice} similarly.  If $\mathcal{L}$ is a meet-semilattice and $x,y \in \mathcal{L}$, we write $x\vee y \in \mathcal{L}$ to mean the statement ``The join of $x$ and $y$ exists in $\mathcal{L}$.''
A poset $\mathcal{L}$ is called a \emph{lattice} if it is a meet-semilattice and a join-semilattice.  For avoiding confusion we may denote $\wedge$ as $\wedge^{\mathcal{L}}$, and  $\vee$ as $\vee^{\mathcal{L}}$.
If $\mathcal{L}$ is a meet-semilattice, we define $\mathcal{L}^+:=\mathcal{L}\backslash \{\hat{0}\}$.  Let $\mathcal{L}$ and $\mathcal{K}$ be meet-semilattices.  A \emph{meet-semilattice order embedding} $f:\mathcal{L}\rightarrow \mathcal{K}$ is an order embedding of $\mathcal{L}$ and $\mathcal{K}$ as posets such that for every $x,y \in \mathcal{L}$, we have $f(x \wedge^{\mathcal{L}} y)=f(x)\wedge^{\mathcal{K}}f(y)$.  

\end{defn}

We emphasize that a meet-semilattice order embedding of lattices need not respect joins.

\begin{rmk}
 If $\mathcal{L}$ is a meet-semilattice and contains a maximum element $\hat{1}$ then $\mathcal{L}$ is a lattice: for any $x,y \in \mathcal{L}$, $$x\vee y := \bigwedge_{x,y\, \leq \, z}z.$$  Thus if $\mathcal{L}$ is a meet-semilattice, or a join-semilattice, and $x,y\in \mathcal{L}$ with $x\leq y$, the interval $[x,y]\subseteq \mathcal{L}$ is itself a lattice. Let $\mathcal{E}$ be a set and $\mathbb{S}$ a lattice of subsets of $\mathcal{E}$ with meet given by intersection, then we refer to $\mathbb{S}$ as an \emph{intersection lattice}.\footnote{We warn the reader that this terminology, although natural, could be a source of confusion as the lattice of flats of a hyperplane arrangement forms an intersection lattice, but this is the reverse of the intersection order  of flats in an arrangement.  In de Concini-Procesi \cite{de1995wonderful}, this confusion is avoided by working in the dual vector space.}
\end{rmk}

\begin{lem}\label{wedgeineqlem}
Let $\mathcal{L},  \mathcal{K}$ be meet-semilattices, $f:\mathcal{L}\rightarrow \mathcal{K}$ a meet-semilattice order embedding, and $x_1, \ldots, x_k \in \mathcal{L}$. 

$${\rm If } \bigvee_{1\leq i \leq k}^{\mathcal{L}}x_i \in \mathcal{L}\,  \,\,\,\,
\text{then}\,\,\,\,\bigvee_{1\leq i \leq k}^{\mathcal{K}}x_i \leq \bigvee_{1\leq i \leq k}^{\mathcal{L}}x_i\,.$$
\end{lem}
\begin{proof}
    Because $f$ is an order embedding, we can identify $\wedge^{\mathcal{K}}=\wedge^{\mathcal{L}}$ and denote this meet by $\wedge$.  Let $\mathcal{S}=\{ y \in \mathcal{K}: \forall\, i \in[k],\, x_i \leq y\}$ and $\mathcal{T}=\{ y \in \mathcal{L}: \forall\, i \in[k] , \, x_i \leq y\}$, then $\mathcal{S}\supseteq \mathcal{T}$.  Hence

    $$\bigvee_{1\leq i \leq k}^{\mathcal{K}}x_i=\bigwedge_{y \in \mathcal{S}}y \leq \bigwedge_{y \in \mathcal{T}}y=\bigvee_{1\leq i \leq k}^{\mathcal{L}}x_i.$$
\end{proof}

\begin{defn}
A lattice $\mathcal{L}$ is \emph{ranked} if there exists a function $r: \mathcal{L}\rightarrow \mathbb{Z}_{\geq 0}$ such that 

\begin{enumerate}
    \item $x \leq y$ in $\mathcal{L}$ implies $r(x) \leq r(y)$.  
    \item If $y$ covers $x$, then $r(y)=r(x)+1$.
    \end{enumerate}
     
A ranked lattice is said to be \emph{supermodular} if $r(x \wedge y) +r(x \vee y) \geq r(x)+r(y)$.

A ranked lattice is said to be \emph{submodular} if $r(x \wedge y) +r(x \vee y) \leq r(x)+r(y)$.
\end{defn}

\begin{defn}\label{main}\cite[Section 2.3]{de1995wonderful}\cite[Definition 2.2]{feichtner2004incidence} Let $\mathcal{L}$ be a meet-semilattice, and $B \subseteq \mathcal{L}^+$.  The set $B$ is a \emph{building set} of $\mathcal{L}$ if for any $x \in \mathcal{L}^+$ and $\text{max}\,B_{\leq x} =\{x_1,...,x_k\}$ there is an isomorphism of posets  
\begin{equation}\label{buildproduct}
    \phi_x: \prod_{j=1}^k[\hat{0},x_j] \rightarrow [\hat{0},x] 
\end{equation}
with $\phi_x(y_1,...,y_k)=y_1 \vee \cdots \vee y_k$ for $y_j \in [\hat{0},x_j]$ with $1\leq j \leq k.$  We call $\text{max}\,B_{\leq x}$ the \emph{set of factors} of $x$ in $B$.  We denote $\mathbb{B}(\mathcal{L})=\{B\subseteq \mathcal{L}^+:B~\text{is a building set of}~\mathcal{L}\}.$  
\end{defn}

\begin{ex}\label{min}
The set of irreducible elements $I(\mathcal{L})$ forms a building set by Lemma \ref{posetfactorlem}. We call $I(\mathcal{L})$ the \emph{minimum building set.}
\end{ex}

\begin{ex}\label{max} The set $\mathcal{L}^+$ is a building set: for any $x\in \mathcal{L}^+$, $ \text{max}\,\mathcal{L}^+_{\leq x}=\{x\}$ and we have the trivial isomorphism $\phi_x: [\hat{0},x] \rightarrow [\hat{0},x]$. We call $\mathcal{L}^+$ the \emph{maximum building set}.
\end{ex}

We note that our terminology in Examples \ref{min} and \ref{max} differs from that of others; the minimum and maximum building sets are typically referred to as the minimal and maximal building sets, respectively.  This revision of terminology is justified by Proposition \ref{lattice}.  

Associated to any building set, one has a distinguished simplicial complex $\mathcal{N}(B)$ called the \emph{nested set complex}.  Typically, when a building set $B$ is utilized for producing a compactification $W_B(X)$ of some space $X$, the complex $\mathcal{N}(B)$ encodes the poset of strata of $W_B(X)$.

\begin{defn}
    Let $\mathcal{L}$ be a meet-semilattice and $B \in \mathbb{B}(\mathcal{L})$.  A subset $N \subseteq B$ is a \emph{nested set} if for every $x_1,...,x_t \in N$ with $t \geq 2$, which are incomparable elements,  $\bigvee_{i=1}^t x_i \in \mathcal{L} \backslash B$.  The \emph{nested set complex} $\mathcal{N}(B)$ is the abstract simplicial complex  consisting of all nested sets of $B$.  For avoiding confusion, we may denote $\mathcal{N}(B)$ as $\mathcal{N}_{\mathcal{L}}(B)$. 
\end{defn}

The primary focus of this article is building sets, but nested set complexes will be discussed in Theorem \ref{1strest}, where we will need the following lemma.  

\begin{lem}\label{nestedlem}\cite[Section 2.4]{de1995wonderful}\cite[Proposition 2.8]{feichtner2004incidence} Let $\mathcal{L}$ be a meet-semilattice, $B \in \mathbb{B}(\mathcal{L})$, and $\mathcal{N}(B)$ the associated nested set complex.  
Let $N \in \mathcal{N}(B)$ with $x_1, \ldots, x_k \in N$ incomparable and let $x = \bigvee_{i=1}^k x_i$.  
Then $\text{max}\, B_{\leq x} = \{x_1,\ldots, x_k\}$.
    
\end{lem}
\begin{proof}
Let $\text{max}\, B_{\leq x} = \{y_1, \ldots, y_l\}$.  For $i$ with $1\leq i\leq l$, let $M_i = N \cap \mathcal{L}_{\leq y_i}$ and $z_i = \bigvee_{x_j \in M_i}x_j$.  Note that $x = \bigvee_{i=1}^l z_i$.  If $\{x_1,\ldots,x_k\} \neq \{y_1, \ldots, y_l\}$ then there exists some $i$ for which $M_i \neq \{y_i\}$. 
 By the definition of a nested set, it follows that $z_i <y_i$, but then because the map $\phi_x$ in Definition \ref{main} is a bijection, we have that $\bigvee_{i=1}^l z_i<\bigvee_{i=1}^l y_i=x$, a contradiction.  Hence $\text{max}\,B_{\leq x}=\{x_1,\ldots,x_k\}$.
\end{proof}
De Concini and Procesi introduced three equivalent characterizations of a building set \cite[Section 2.3]{de1995wonderful}.  Their third and most concise characterization has been featured prominently for the Boolean lattice in the work of Postnikov \cite[Definition 7.1]{ postnikov2009permutohedra} and Feichtner-Sturmfels \cite[Lemma 3.9]{feichtner2005matroid}\footnote{We note that building sets for the Boolean lattice have also appeared in the work of Schmitt under the name Whitney systems \cite[Definition 1]{schmitt1995hopf}.}\,\footnote{Building sets for the Boolean lattice correspond to certain smooth polytopes now known as \emph{nestohedra}.}, but this characterization has been largely overlooked by researchers in matroid theory.  Recently, Bibby-Denham-Feichtner proved that one half of this characterization holds for locally geometric semilattices \cite[Proposition 2.5.3 (b)]{bibby2022leray}, a result which was utilized by Pagaria-Pezzoli \cite{pagaria2023hodge}. Here we extend de Concini-Procesi's characterization to general meet-semilattices and note that this appears to be the first time this characterization has been described in full for geometric lattices.  To recover the Boolean case from Proposition \ref{usefulchar}, note that here $I(\mathcal{L})=\{\{i\}:i\in [n]\}$, meets and joins are intersections and unions, respectively, and joins always exist.

\begin{prop}\label{usefulchar}
Let $\mathcal{L}$ be a meet-semilattice and $B$ a subset of $\mathcal{L}^+$.  Then $B$ is a building set if and only if
\begin{enumerate}
    \item\label{bscon1} $I(\mathcal{L})\subseteq B$
    \item\label{bscon2} if $x,y \in B$ such that $x \wedge y \neq \hat{0}$ and $x \vee y \in \mathcal{L}$, then $x \vee y \in B$.
\end{enumerate} 
\end{prop}

\begin{proof}
($\Rightarrow$)  Let $B$ be a building set of  $\mathcal{L}$.  It is clear from the definition of a building set that $I(\mathcal{L})\subseteq B$. 
Let $x,y \in B$  such that $x \vee y \notin B$.  If $x \vee y \notin \mathcal{L}$, there is nothing to prove, so suppose that $x \vee y \in \mathcal{L}$,
and let $\text{max}\,B_{\leq x \vee y}=\{z_1,...,z_k\}.$  Both $x$ and $y$ are strictly less than $x\vee y$ as $x\vee y \not\in B$.  Therefore $x\in [\hat{0},z_s]$ and $y \in [\hat{0},z_t]$ where $s \neq t$.  Since $B$ is a building set there is an isomorphism $\phi_{x \vee y}: \prod_{j=1}^k[\hat{0},z_j] \rightarrow [\hat{0},x\vee y]$.   Because $\phi_{x \vee y}$ is injective, the intersection of these intervals is $\hat{0}$. It follows that $x \wedge y=\hat{0}$.  
\\
$(\Leftarrow)$ Let $B\subseteq \mathcal{L}^+$ which satisfies conditions (\ref{bscon1}) and (\ref{bscon2}).  Let $x \in \mathcal{L}^+$ and let $\text{max}\,B_{\leq x}=\{x_1,...,x_k\}$.  We wish to show  $\prod_{j=1}^k[\hat{0},x_j] \cong [\hat{0},x]$.   We begin with the observation that for each $1\leq i,j \leq k$ with $i\neq j$, $[\hat{0},x_i]\cap [\hat{0},x_j] = \hat{0}$.  Suppose this is not the case and let $z \in [\hat{0},x_i]\cap [\hat{0},x_j]$ with $\hat{0}\neq z$.   Then $x_i \wedge x_j \neq \hat{0}$ and $x_i,x_j < x$, so $x_i \vee x_j \in \mathcal{L}$, thus $x_i \vee x_j \in B$. But $x_i,x_j <x_i \vee x_j$ and $x_i \vee x_j \leq x$ which contradicts the maximality of $x_i$ and $x_j$.

Let $\text{max}\,I(\mathcal{L})_{\leq x}=\{y_1,\ldots,y_l\}$.  For each $1\leq t\leq k$ let $\pi_t=\{y_i : y_i \leq x_t\}$.  We claim that the set of $\pi_t$ forms a partition of $\text{max}\,I(\mathcal{L})_{\leq x}$.  By assumption, $I(\mathcal{L})$ is a subset of $B$, so for each $y_i$ there exists some $x_t$ such that $y_i \leq x_t$, and we know by disjointness of the $[\hat{0},x_i]$, that $x_t$ is unique in this regard -- this establishes the partition.

We next claim that $\text{max}\,I(\mathcal{L})_{\leq x_t}=\pi_t$.  Let $y \in   \text{max}\,I(\mathcal{L})_{\leq x_t}$, and let $y_i$ be such that $y \leq y_i$.  By our previous argument, there exists a unique $t$ such that $y_i \leq x_t$, hence $y = y_i$.
By Lemma \ref{posetfactorlem} there are isomorphisms 
\[
\phi_x: \prod_{j=1}^l [\hat{0},y_j] \xrightarrow{\cong} [\hat{0},x] \hspace{.5cm} \text{and}\hspace{.5cm} \phi_{x_t}: \prod_{y \in \pi_t} [\hat{0},y] \xrightarrow{\cong} [\hat{0},x_t].
\] 
It follows that $\phi_x \circ \prod_{t=1}^k (\phi_{x_t})^{-1}:\prod_{t=1}^k[\hat{0},x_t]\rightarrow [\hat{0},x]$ is an isomorphism.

\end{proof}

We will now investigate the poset $(\mathbb{B}(\mathcal{L}),\subseteq)$. When $\mathcal{L}$ is clear from context we may simply refer to $(\mathbb{B}(\mathcal{L}),\subseteq)$ as $\mathbb{B}$.  The following observation extends a result of Feichtner-Sturmfels \cite[Lemma 3.10]{feichtner2005matroid} from the Boolean lattice to an arbitrary meet-semilattice.

\begin{prop}\label{lattice}
Let $\mathcal{L}$ be a meet-semilattice.   The poset $(\mathbb{B}(\mathcal{L}),\subseteq)$ is an intersection  lattice. 
\end{prop}
\begin{proof}
 We first utilize Proposition \ref{usefulchar} for showing that the intersection of two building sets is a building set.  Let $B$ and $B'$ be building sets. As $I(\mathcal{L}) \subseteq B,B'$, we have that $I(\mathcal{L}) \subseteq B \cap B'$.  Let $x$ and $y$ be elements of $B \cap B'$ such that $x \wedge y \neq \hat{0}$.  If $x \vee y \in 
 \mathcal{L}$, then $x \vee y \in B$ and $x \vee y \in B'$ as $B$ and $B'$ are building sets.  Therefore $x \vee y \in B \cap B'$.  The poset  $\mathbb{B}$ is a meet-semilattice with a maximum element $\mathcal{L}^+$, hence it is a lattice.
\end{proof}

\section{Convex Geometries}

In this section we review the theory of convex geometries and prove Proposition \ref{super}, which states that the collection of building sets on a meet-semilattice forms a supersolvable convex geometry.  We provide a new characterization of supersolvable convex geometries and conclude with a brief discussion of how one can algorithmically produce building sets.

\begin{defn}
A \emph{set system} is a pair $(\mathcal{E},\mathbb{S})$ where $\mathcal{E}$ is a set and $\mathbb{S} \subseteq 2^{\mathcal{E}}$.  We refer to $\mathcal{E}$ as the \emph{ground set}.
\end{defn}
\begin{defn}\label{closure}
A \emph{closure operator} on a ground set $\mathcal{E}$ is a map $\sigma: 2^{\mathcal{E}} \rightarrow 2^{\mathcal{E}}$ such that for each $A , B \subseteq \mathcal{E}$,  
\begin{enumerate}
    \item\label{closure1}  $A \subseteq \sigma(A)$ 

    \item\label{closure2} $A \subseteq B \Rightarrow \sigma(A) \subseteq \sigma(B)$ 

    \item\label{closure3} $\sigma(\sigma(A))=\sigma(A)$ 
\end{enumerate}

 We call the set $A$ \emph{closed} if $\sigma(A)=A$.
\end{defn}

\begin{defn}\label{closuresystem}
Let $\sigma$ be a closure operator on ground set $\mathcal{E}$.  Let $\mathbb{S}=\{A \in 2^{\mathcal{E}}:\sigma(A)=A\}$ be the closed subsets of $\mathcal{E}$.  Then $(\mathcal{E}, \mathbb{S}, \sigma)$ is called a \emph{closure system}. 
\end{defn}

\begin{lem}\label{capclosed}
Let $(\mathcal{E}, \mathbb{S}, \sigma)$ be a closure system and $A,B \subseteq \mathcal{E}$, then $\sigma(A) \cap \sigma(B) \in \mathbb{S}$.
\end{lem}
\begin{proof}
It needs to be shown that $\sigma(\sigma(A) \cap \sigma(B))=\sigma(A) \cap \sigma(B)$.  By (\ref{closure2}) of Definition \ref{closure}, $\sigma(\sigma(A) \cap \sigma(B)) \subseteq \sigma(\sigma(A))=\sigma(A)$ by (\ref{closure3}) and $\sigma(\sigma(A) \cap \sigma(B)) \subseteq \sigma(\sigma(B))=\sigma(B)$ by (\ref{closure3}),  which implies $\sigma(\sigma(A) \cap \sigma(B))\subseteq \sigma(A) \cap \sigma(B)$.  The reverse inclusion is a consequence of (\ref{closure1}) of Definition \ref{closure}.
\end{proof}

\begin{rmk}  
Every closure system $(\mathcal{E},\mathbb{S},\sigma)$ comes with a natural meet and join: for $A,B \in \mathbb{S}$,  $A \wedge B=A \cap B$ and $A \vee B=\sigma(A \cup B)$.  Thus $\mathbb{S}$ is an intersection lattice.  Conversely, given an intersection lattice $\mathbb{S}$ with ground set $\mathcal{E}$, we can define a closure operator $\sigma: 2^{\mathcal{E}} \rightarrow 2^{\mathcal{E}}$ as follows: given $X \subseteq \mathcal{E}$, let $$\sigma(X) = \bigcap_{A \in \mathbb{S},X \subseteq A}A.$$  This induces a canonical identification of intersection lattices and closure operators.  It follows, in particular, that a closure system $(\mathcal{E},\mathbb{S},\sigma)$ is encoded by $(\mathcal{E},\mathbb{S})$ alone, which we may denote by $\mathbb{S}$ when $\mathcal{E}$ is clear from the context.
\end{rmk}

It was observed by Feichtner-Sturmfels \cite[Lemma 3.10]{feichtner2005matroid} that in the case of the Boolean lattice, Proposition \ref{lattice} implies that the collection of building sets determines a closure operator.  Thus far, our article has been concerned with extending results of de Concini-Procesi \cite{de1995wonderful}, Postnikov \cite{postnikov2009permutohedra}, and Feichtner-Sturmfels \cite{feichtner2005matroid} to Feichtner-Kozlov's setting of general meet-semilattices \cite{feichtner2004incidence}.  We now go beyond the works of those authors, even in the case of the Boolean lattice, in understanding the structure of building sets.

A convex geometry is a special type of closure operator introduced independently by Edelman and Jamison.

\begin{defn}\label{antiex}\cite{edelman1980meet,jamison1982perspective}
Let $(\mathcal{E}, \mathbb{S}, \sigma)$ be a closure system and assume in addition that $\sigma$ satisfies
\[ \bullet\,\,
\text{for every}~ A \in \mathbb{S}~\text{and any distinct}~x,y\not\in A,\text{if}\,~y \in \sigma(A \cup \{x\}) \,\text{then}\, x \not\in \sigma(A \cup \{y\}).
\]
Then $\sigma$ is an \emph{anti-exchange closure operator} and $(\mathcal{E}, \mathbb{S}, \sigma)$ is a \emph{convex geometry}.  We may denote the convex geometry by $(\mathcal{E}, \mathbb{S})$, or simply $\mathbb{S}$, when no confusion will arise.
\end{defn}

 Greedoids are a generalization of matroids which were originally introduced by Korte and Lov\'asz  in the setting of greedy algorithms \cite{korte1981mathematical}.  Antimatroids form another important class of greedoids.  They first appeared in the work of Dilworth \cite{dilworth1940lattices} where they were described in the language of lattices.  Greedoids, antimatroids, and convex geometries, much like matroids, admit several cryptomorphic definitions.  We refer the reader to the excellent survey by Bj\"orner and Ziegler for an introduction to these topics \cite{bjorner1992introduction}.  The following definition is one of the most well-known definitions of an antimatroid. 

\begin{defn}
An \emph{antimatroid} is a set system $(\mathcal{E}, \mathbb{S})$ such that 
\begin{enumerate}
    \item $\emptyset \in \mathbb{S}$ 
    \item $\text{for every}~ A,B \in \mathbb{S}~\text{such that}~B \nsubseteq A,~\text{there exists}~e \in B\backslash A~\text{such that}~A \cup \{e\} \in \mathbb{S}$.
\end{enumerate}
\end{defn}

   For any set system $(\mathcal{E},\mathbb{S})$, there is a complementary set system  $(\mathcal{E},\mathbb{S}^c)$ where $\mathbb{S}^c=\{\mathcal{E}\backslash F:F \in \mathbb{S}\}$.  A fundamental result of Bj\"orner states that a set system $(\mathcal{E},\mathbb{S})$ is a convex geometry  if and only if the complementary set system $(\mathcal{E},\mathbb{S}^c)$ is an antimatroid \cite{bjorner1983matroids}.  Thus a convex geometry can alternately be defined as follows.

\begin{defn}\label{super2}
A \emph{convex geometry} is a set system $(\mathcal{E}, \mathbb{S})$ such that 
\begin{enumerate}
    \item\label{super2i}   $\mathcal{E} \in \mathbb{S}$ 

    \item\label{super2ii}  $\text{for every}~ A,B \in \mathbb{S}~\text{such that}~B \nsupseteq A,~\text{there exists}~e \in A\backslash B~\text{such that}~A \backslash \{e\} \in \mathbb{S}$ 
\end{enumerate}
\end{defn}

We note the following important fact.

\begin{prop}\label{convexgeometriesarelattices}
If $(\mathcal{E}, \mathbb{S})$ is a convex geometry then $\mathbb{S}$ is an intersection lattice.
\end{prop}

\begin{proof}
    Let $A,B \in \mathbb{S}$.  By repeated application of Definition \ref{super2} condition (\ref{super2ii}), we may remove all of the elements in $A\setminus B$ from $A$ to produce $A\cap B \in \mathbb{S}$.  Therefore, $\mathbb{S}$ is an intersection meet-semilattice with a maximum element $\mathcal{E}$, hence $\mathbb{S}$ is an intersection lattice.
\end{proof}

Supersolvable lattices were first introduced by Stanley \cite[Definition 1.1]{stanley1972supersolvable}, motivated by the study of supersolvable groups.  In the same work, he defined  supersolvable join-distributive lattices.  Edelman proved that join-distributive lattices are cryptomorphic to antimatroids \cite[Theorem 3.3]{edelman1980meet}.\footnote{More accurately, Edelman proved that meet-distributive lattices are cryptomorphic to convex geometries.}  Hawrylycz and
Reiner demonstrated that the collection of all closure operators on a poset defines a supersolvable join-distributive lattice \cite[Theorem 10]{hawrylycz1993lattice}.  Armstrong provided a natural cryptomorphic description of supersolvable join-distributive lattices in the context of antimatroids, and applied this notion to the study of the sorting order of a Coxeter group \cite{armstrong2009sorting}.  

\begin{defn}\label{antim}\cite[Definition 2.11]{armstrong2009sorting}
    Let $(\mathcal{E},\mathbb{S})$ be a  set system. 
    Let $< $ be a total order of $\mathcal{E}$.  We say that $\mathbb{S}$ is a \emph{supersolvable antimatroid} if 
\begin{enumerate}

\item\label{antim1} $\emptyset \in \mathbb{S}$

\item\label{antim2} $\text{for every}~ A,B \in \mathbb{S} ~\text{with}~ B \nsubseteq A,~\text{for}~e =\text{min}_{<}(B\backslash A) ~\text{we have }~ A \cup \{e\} \in \mathbb{S}.$
\end{enumerate}

We say $\mathbb{S}$ is supersolvable with respect to $<$.
\end{defn}
Thus one has the following dual definition of supersolvable convex geometries, which was utilized by Gillespie in the study of matroid activities \cite[Proposition 15]{gillespie2020convexity}. 

\begin{defn}\label{supercg}\cite{armstrong2009sorting} \cite{gillespie2020convexity}\,\footnote{Definition \ref{supercg} and the phrase ``supersolvable convex geometry'' do not appear in \cite{armstrong2009sorting,gillespie2020convexity}, nonetheless we feel it is appropriate to attribute this definition to those sources as it is implicit in their discussions.} 
    Let $(\mathcal{E},\mathbb{S})$ be a set system. 
    Let $< $ be a total order of $\mathcal{E}$.  We say that $\mathbb{S}$ is a \emph{supersolvable convex geometry} if 
\begin{enumerate}

\item\label{supercg1} $\mathcal{E} \in \mathbb{S}$

\item\label{supercg2}  $\text{for every}~ A,B \in \mathbb{S} ~\text{with}~ B \nsupseteq A,~\text{for}~e =\text{min}_{<}(A\backslash B) ~\text{we have }~A\setminus \{e\} \in \mathbb{S}.$
\end{enumerate}

We say $\mathbb{S}$ is supersolvable with respect to $<$.
\end{defn}

De Concini-Procesi proved the following statement in the case of subspace arrangements, and van der Veer showed the result for meet-semilattices\footnote{Van der Veer utilized this result for proving that the motivic zeta function of a hyperplane arrangement is independent of the building set.}.

\begin{prop}\label{smaller}\cite[Section 2.5]{de1995wonderful}\cite[Lemma 2]{van2019combinatorial}
 Let $\mathcal{L}$ be a meet-semilattice and $B \in \mathbb{B}(\mathcal{L})$.  Let $<$ be a linear extension of $\mathcal{L}$, then for $x = \text{min}_{<}(B\backslash I(\mathcal{L}))$ we have $B\setminus \{x\} \in \mathbb{B}(\mathcal{L})$.
\end{prop}

Feichtner-M\"uller established the following strengthened version.

\begin{prop}\label{FM}\cite[proof of Theorem 4.2]{feichtner2005topology}
Let $\mathcal{L}$ be an atomic meet-semilattice, and $B,B' \in \mathbb{B}(\mathcal{L})$ with $B' \subsetneq B$.  Let $<$ be a linear extension of $\mathcal{L}$, then for $x = \text{min}_{<}(B\backslash B')$ we have $B\setminus \{x\} \in \mathbb{B}(\mathcal{L})$.
\end{prop}

We offer the following extension of Proposition \ref{FM}.  We note that by Proposition \ref{convexgeometriesarelattices}, the closed sets of a convex geometry form an intersection lattice, hence Proposition \ref{super} subsumes Proposition \ref{lattice}.  This result can be equivalently interpreted as an extension of Proposition \ref{FM} to the setting where $B$ and $B'$ may not be comparable, or as a combination of Proposition \ref{FM} and Proposition \ref{lattice}.

\begin{prop}\label{super}
Let $\mathcal{L}$ be a meet-semilattice, then $(\mathcal{L}^+,\mathbb{B}(\mathcal{L}))$ is a supersolvable convex geometry which is supersolvable with respect to any linear extension of $\mathcal{L}$.
\end{prop}

\begin{proof}
Let $<$ be a linear extension of $\mathcal{L}$. As was noted in example \ref{max}, $\mathcal{L}^+ \in \mathbb{B}$ thus condition (\ref{supercg1}) is satisfied. To verify condition (\ref{supercg2}), let $B,B' \in \mathbb{B}$ with $B'\not\subseteq B$, and $x = \text{min}_{<}(B'\backslash B)$.  We claim that $B'\backslash \{x\} 
\in 
\mathbb{B}$, and we verify the conditions of Proposition \ref{usefulchar}.
Since $I(\mathcal{L}) \subseteq B$ and $x \not \in B$, $x$ is not irreducible so $B'\backslash \{x\}$ contains $I(\mathcal{L})$.  Let $y,z \in B'\backslash \{x\}$ such that $y \wedge z \neq \hat{0}$. If $y \vee z \in \mathcal{L}$, by Proposition \ref{usefulchar}, $y \vee z \in B'$.  We wish to show that $y \vee z \neq x$ for which we need only consider $y,z <x$. By the minimality of $x$, we know $y,z \in B$. By Proposition \ref{usefulchar}, $y \vee z \in B$ so certainly $y \vee z \neq x$.  Thus $y \vee z \in B'\backslash \{x\}$.  By Proposition \ref{usefulchar}, it follows that $B'\backslash \{x\} \in \mathbb{B}$. 
\end{proof}

\begin{cor}
$\mathbb{B}(\mathcal{L})$ is a supermodular lattice.  
\end{cor}
\begin{proof}
This is a standard property of any convex geometry -- we provide a proof here.  For $B \in \mathbb{B}(\mathcal{L})$, set $B_{\circ}:=B\backslash I(\mathcal{L})$.  Define the function $r:\mathbb{B}\rightarrow \mathbb{Z}_{\geq 0}$ by $r(B)= |B_{\circ}|$, and note that $r(I(\mathcal{L}))=0$. Let $B,B'\in \mathbb{B}$ such that $B'$ covers $B$.  By Proposition \ref{super}, if $x = \text{min}_<(B'\setminus B)$, then $B'\backslash \{x\} \in \mathbb{B}$, thus $B' = B \cup \{x\}$, and $r(B')=r(B)+1$. 
 Therefore $\mathbb{B}(\mathcal{L})$ is ranked.  To see that $r$ is supermodular, let $B,B' \in \mathbb{B}$.  Then $
r(B \vee B')=|(B\vee B')_{\circ}|\geq|B_{\circ}\cup B'_{\circ}|
$ and 
 $
r(B \wedge B')=|(B \wedge B')_{\circ}|=|(B \cap B')_{\circ}|=|B_{\circ} \cap B'_{\circ}|$ together imply
\[
r(B \vee B')+r(B \wedge B')\geq |B_{\circ}\cup B'_{\circ}|+|B_{\circ} \cap B'_{\circ}|=|B_{\circ}|+ |B'_{\circ}|=r(B)+ r(B').
\]
\end{proof}

  \begin{defn}
    Let $(\mathcal{E}, \mathbb{S}, \sigma)$ be a closure system, and let $A \subseteq \mathcal{E}$.  An element $x \in A$ is called an \emph{extreme point} of $A$ if $x \notin \sigma(A \setminus \{x\})$.  We denote the set of extreme points of $A$ by $\text{ex}(A)$.
\end{defn}

We have the following characterization of the extreme points of a building set.

\begin{prop}\label{extremebuild}
Let $\mathcal{L}$ be a meet-semilattice and  $B \in \mathbb{B}(\mathcal{L})$.  Then $x \in \text{ex}(B)$ if and only if 

\begin{enumerate}
    \item $x \notin I(\mathcal{L})$ and 
    \item if $y,z \in B\setminus \{x\}$ such that $y \vee z = x$, then $y \wedge z =\hat{0}$.
\end{enumerate}  
\end{prop}
\begin{proof}
 This is a direct consequence of Proposition \ref{usefulchar}.  
\end{proof}

We now introduce a new characterization of supersolvable convex geometries, which we feel is a bit closer to the definition of a building set.

\begin{defn}
    Let $\sigma: 2^{\mathcal{E}} \rightarrow 2^{\mathcal{E}}$ be a closure operator and $<$ be a total order on $\mathcal{E}$.  We say that $\sigma$ is a \emph{supersolvable closure operator} (with respect to $<$) if for each closed set $A$ and $e \in \mathcal{E}\setminus A$, we have that 
    
    \begin{itemize}
    \item if $f \in \sigma(A \cup \{e\})\setminus (A \cup \{e\})$ then $e<f$.
    \end{itemize}
\end{defn}

\begin{prop}\label{altsuper}
 Let $(\mathcal{E},\mathbb{S})$ be a set system, then $\mathbb{S}$ is a supersolvable convex geometry if and only if $\mathbb{S}$ is the collection of closed sets of a supersolvable closure operator.
\end{prop}

\begin{proof}
    Suppose that $(\mathcal{E},\mathbb{S})$ is a supersolvable convex geometry with associated closure operator $\sigma$.   Suppose that $\sigma$ is not supersolvable, then there exists a closed set $A$ and $e,f \in \mathcal{E}\setminus A$ such that $f \in \sigma(A \cup \{e\})\setminus (A \cup \{e\})$, and $f< e$.  Take $f$ to be the minimum such element with respect to $<$. Applying condition (\ref{supercg2}) of  Definition \ref{supercg} to the pair of sets $A, \sigma(A \cup \{e\})$, we have that $\sigma(A \cup \{e\}) \setminus \{f\} \in \mathbb{S}$, but this contradicts the definition of $\sigma(A \cup \{e\}) $.  
    
    Conversely, suppose that $\sigma$ is a supersolvable closure operator on $\mathcal{E}$ with $\mathbb{S}$ the associated closed sets.  
    Let $A,B \in \mathbb{S}$ and let $e = \text{min}_{<}(B\setminus A)$.  Suppose for contradiction that $B\setminus \{e\}$ is not a closed set.  As $\sigma$ is a closure operator, we know by Lemma \ref{capclosed} that $A\cap B$ is closed.  Let $e=e_0 < \ldots<e_k$ be the elements of $B\setminus A$.  Let $i$ be the smallest index such that $e_0 \in \sigma((A\cap B) \cup \{e_1,\ldots, e_i\})$.  Let $X = \sigma((A\cap B) \cup \{e_1,\ldots, e_{i-1}\})$, then $e_0 \in \sigma(X \cup\{e_i\})\setminus (X\cup\{e_i\})$, but $e_0 < e_i$ contradicting supersolvability of $\sigma$.
\end{proof}

 From Proposition \ref{altsuper}, it follows that one of the most fundamental examples of a convex geometry, the upper ideals in a poset, form a supersolvable convex geometry with respect to any linear extension of the underlying poset.\footlabel{oderfilt}{This observation is related to Adiprasito-Huh-Katz's inductive proof of the K\"ahler package for matroids \cite{adiprasito2018hodge}, where they define certain fans induced by upper ideals (order filters) in the lattice of flats.  We note that while an upper ideal in a meet-semilattice $\mathcal{L}$ is typically not a building set, it becomes a building set after unioning with $I(\mathcal{L})$, although for the lattice of flats of a matroid, the associated Bergman fan is not the one considered in \cite{adiprasito2018hodge}.}

Proposition \ref{super} and Proposition \ref{altsuper} demonstrate that the closure operator for building sets is supersolvable with respect to any linear extension of $\mathcal{L}$.  We will now apply Proposition \ref{usefulchar} for computing $\sigma(X)$, which in turn combined with Proposition \ref{altsuper} gives an alternate proof of Proposition \ref{super}. 

\begin{prop}\label{closurealgo}
 Let $\mathcal{L}$ be a meet-semilattice and $X \subseteq \mathcal{L}^+$.  Set $X_0:=X \cup I(\mathcal{L})$.  Given $X_i$, set $Y_i := \{a\vee b \in \mathcal{L} : a,b \in X_i\,,\, a\wedge b \neq \hat{0}\}$ and $X_{i+1} := X_i \cup Y_i.$  Let $k$ be such that $X_{k+1}=X_k$, then $\sigma(X) = X_k$.
\end{prop}

Note that because $\mathcal{L}$ is finite and $X_i\subseteq X_{i+1}$, there exists some $k$ such that $X_{k+1}=X_k$.

\begin{proof}
By Proposition \ref{usefulchar}, it follows that $X_k$ is a building set, so $\sigma(X) \subseteq X_k$.  We establish the opposite inclusion by induction.  Clearly, $X_0 \subseteq \sigma(X)$.  Suppose for some $i$ we have that $X_i \subseteq \sigma(X)$, then $Y_i \subseteq \sigma(X)$, hence $X_{i+1} \subseteq \sigma(X)$.  
\end{proof}

We briefly describe a general method to produce building sets. First, take some $X,X'\subseteq \mathcal{L}^+$.  Construct $\sigma(X),\sigma(X')$ using Proposition \ref{closurealgo}. If $\sigma(X) \nsubseteq \sigma(X')$, we can apply supersolvability to remove elements from $\sigma(X')\setminus \sigma(X)$ one at a time.  All of the intermediate objects will be building sets.

\section{Restrictions of building sets and nested set complexes}

In this section we utilize the perspective developed in this article for proving Theorem \ref{1strest} which states that building sets and nested set complexes behave well with respect to certain restrictions of meet-semilattices.  Our theorem is inspired by the following wonderful result of Mantovani-Padrol-Pilaud \cite{mantovani2025facial} for oriented matroids, which is motivated by fundamental polyhedral considerations.\footnote{We remark that the recent work of Brauner-Eur-Pratt-Vlad \cite{brauner2024wondertopes} has some overlap with the work of Mantovani-Padrol-Pilaud \cite{mantovani2025facial}, but not in this particular aspect it seems.}
\begin{thm}\label{ormatrest}\cite{mantovani2025facial}
Suppose that $(E,\mathcal{M})$ is a simple acyclic oriented matroid with Las Vergnas' face lattice $\mathcal{F}$, and let $\mathcal{B}_n=(2^{E},
\subseteq)$ be the Boolean lattice where $|E|=n$.  Let $\mathcal{F} \hookrightarrow \mathcal{B}_n$ be the natural embedding. Then

\begin{enumerate}
    \item $\mathbb{B}(\mathcal{F}) = \{B \cap \mathcal{F}: B \in \mathbb{B}(\mathcal{B}_n), I(\mathcal{F}) \subseteq B\}.$
    \item If $B \in \mathbb{B}(\mathcal{B}_n)$ and  $I(\mathcal{F}) \subseteq B$, then   $\mathcal{N}(B\cap \mathcal{F})\subseteq \mathcal{N}(B)$.
\end{enumerate}

\end{thm}

Our primary goal in proving Theorem \ref{1strest} is to establish an unoriented matroid version of Theorem \ref{ormatrest}, which has a separate polyhedral motivation coming from Hodge theory for matroids.\footnote{
 We originally learned the statement of Theorem \ref{ormatrest} from Padrol's lecture at the Simons Center for Geometry and Physics Workshop on ``Combinatorics and Geometry of Convex Polyhedra" in March 2023.  In personal communication, we asked Mantovani-Padrol-Pilaud whether the natural unoriented version of Theorem \ref{ormatrest} should hold, and they said that they had considered, but not investigated this question.  After an earlier draft of this work appeared on arXiv, Matt Larson contacted us to ask whether the unoriented version of Theorem \ref{ormatrest} (now Corollary \ref{maincor}) follows from the results of this paper. Inspired by this independent interest, the authors verified this statement which resulted in  Theorem \ref{1strest} -- we thank Matt Larson for his question.   We thank  Arnau Padrol and Vincent Pilaud for their encouragement with the result of this section, and for giving us their blessing to present it here.  Update from arXiv version 4: The work of Mantovani-Padrol-Pilaud is now available \cite{mantovani2025facial}.}  Although we will not define an oriented matroid, which is not our focus here, we will now define a matroid.  Because the perspective of this article is lattice theoretic, we provide the definition of a matroid in terms of its lattice of flats.

\begin{defn}
    A \emph{matroid} $M$ is an intersection lattice  $(E,\mathcal{L})$  
    such that 
    \begin{itemize}[leftmargin=9pt]
    \item for every $F \in \mathcal{L}$ and $x \in E\setminus F$, there is a unique $G \in \mathcal{L}$ which covers $F$ such that $x \in G$.
    \end{itemize}
The elements of $\mathcal{L}$ are called \emph{flats}. 
\end{defn}

 A matroid is \emph{simple} if each atom of $\mathcal{L}$ is a singleton from $E$.  Given matroids $M_1 =(E,\mathcal{L}_1)$ and $M_2 =(E,\mathcal{L}_2)$, we say that $M_1$ is a \emph{quotient} of $M_2$ if $\mathcal{L}_1 \subseteq \mathcal{L}_2$.  We now recall the definition of a Bergman fan of a matroid with respect to a building set.  
 
 \begin{defn}
     
 Let $M$ be a matroid on ground set $E$ with $|E| =n$.  Let $\mathcal{L}$ be the lattice of flats of $M$ and let $B$ be a building set on $\mathcal{L}$ with $E \in B$\footnote{Some authors do not require building sets for matroids to contain the maximum element when discussing Bergman fans (see \cite{eur2025building, mantovani2025facial}).  Our result readily extends to that level of generality.}.  The \emph{Bergman fan} of $M$ with respect to $B$ is the following collection of polyhedral cones in $\mathbb{R}^{n}\,$\text{:}

$$\Sigma = \{ cone\{\chi_F:F \in N\}+\langle {\bf 1}\rangle_{\mathbb{R}}:N \in \mathcal{N}(B)\},$$

\

where $\chi_F$ is the indicator vector for the flat $F$, and $\langle{\bf 1}\rangle_{\mathbb{R}}$ is the real span of the all ones vector.  If $M$ is the free matroid, i.e. $\mathcal{L}$ is the Boolean lattice, we refer to $\Sigma$ as a \emph{nestohedral fan}.
\end{defn}

Wonderful compactifications of the torus (the complement of the coordinate arrangement) are smooth projective toric varieties, and nestohedral fans are the complete fans associated to these varieties.\footnote{This correspondence holds after first quotienting the Bergman fan by its lineality space $\langle {\bf 1}\rangle_{\mathbb{R}}$.}  Direct combinatorial proofs that nestohedral fans are smooth and projective are given in Postnikov \cite[Proposition 7.10]{postnikov2009permutohedra} and Zelevinsky \cite[Corollary 5.2]{zelevinsky2006nested}.
 Feichtner-Sturmfels \cite[Proposition 3.6]{feichtner2005matroid} prove  that for arbitrary $M$ and $B$, the Bergman fan $\Sigma$ forms a unimodular, but typically incomplete, fan whose support is independent of the choice of $B$.
 The Bergman fan with respect to the maximum building set, as introduced in Ardila-Klivans \cite{ardila2006bergman}, is a subfan of the braid arrangement, i.e. the nestohedral fan with respect to the maximum building set.  Several researchers have presented interesting nested set complex structures for Bergman fans as restrictions of nestohedral fans: Braden-Huh-Matherne-Proudfoot-Wang describe the augmented Bergman fan as a subfan of the stellahedral fan \cite[Proposition 2.6]{braden2022semi},  Crowley-Huh-Larson-Simpson-Wang describe the polymatroid Bergman fan as a subfan of the polypermutahedral fan \cite[Definition 1.7]{crowley2022bergman}, and Eur-Larson describe the augmented Bergman fan of a polymatroid as a subfan of the polystellahedral fan \cite[Definition 3.7]{eur2023intersection}.

We offer a unification and generalization of the above results.  

\begin{defn}
Let $\mathcal{L}$ and $\mathcal{K}$ be meet-semilattices and let $\mathcal{L} \hookrightarrow \mathcal{K}$ be a meet-semilattice order embedding. We say that the embedding $\mathcal{L} \hookrightarrow \mathcal{K}$ is \emph{consistent} if 

\begin{itemize}
    \item for every $x \in \mathcal{L}^+$, we have  $I(\mathcal{K})_{\leq_{\mathcal{K}} x} \subseteq I(\mathcal{L})_{\leq_{\mathcal{L}} x}$.
\end{itemize}

\end{defn}

Observe that if $\hookrightarrow$ is a consistent meet-semilattice order embedding, then $\hat{0}_{\mathcal{L}}=\hat{0}_{\mathcal{K}}$, which we may denote by $\hat{0}$.  The following is the main result of this section.

\begin{thm}\label{1strest}
    Let $\mathcal{L}$ and $\mathcal{K}$ be meet-semilattices and let $\mathcal{L} \hookrightarrow \mathcal{K}$ be a consistent meet-semilattice order embedding. Then
    
    \begin{enumerate}
    \item\label{1stpart} $\mathbb{B}(\mathcal{L})=\{B\cap 
    \mathcal{L}^+:B\in \mathbb{B}(\mathcal{K})\, \,,\,I(\mathcal{L}) \subseteq B \}.$
    \item\label{2ndpart}  If $B \in  \mathbb{B}(\mathcal{K})$ and $I(\mathcal{L}) \subseteq B$, then $\mathcal{N}_{\mathcal{L}}(B \cap \mathcal{L}^+)\subseteq \mathcal{N}_{\mathcal{K}}(B).$
    \end{enumerate}

\end{thm}

\

Let $\mathcal{L}$ be an intersection lattice on the ground set $[n]$ such that each atom is a singleton. Take $\mathcal{K}=\mathcal{B}_n$, the Boolean lattice, 
 and $\mathcal{L}\hookrightarrow \mathcal{B}_n$ the natural embedding.  Then $\hookrightarrow$ is a consistent meet-semilattice order embedding.  Thus we have the following Corollary \ref{maincor}.

\begin{cor}\label{maincor}
    Let $M$ be a simple matroid with lattice of flats $\mathcal{L}$ and take $B$ a building set for $\mathcal{L}$ with $E\in B$.    Let  $\Sigma$ be the Bergman fan associated to $M$ and $B$.  There exists a nestohedral fan $\Sigma'$ such that $\Sigma$ is a subfan of $\Sigma'$.
\end{cor}

\begin{proof}[Proof of Theorem \ref{1strest}]
(\ref{1stpart}): Because $\hookrightarrow$ is a meet-semilattice order embedding, we can identify $\leq_{\mathcal{K}}$ with $\leq_{\mathcal{L}}$, and $\wedge^{\mathcal{K}}$ with $\wedge^{\mathcal{L}}$, which we will denote by $\leq$ and $\wedge$, respectively, throughout the proof.   Let $\sigma$ be the building set closure operator associated to $\mathcal{K}$.   We begin by proving that  $\mathbb{B}(\mathcal{L})= \{ \sigma(B)\cap\mathcal{L}^+ :B \in \mathbb{B}(\mathcal{L})\}.$  Let 
$\mathbb{X} = \{ B \in \mathbb{B}(\mathcal{L}):\sigma(B)\cap\mathcal{L}^+ = B\}.$  
We would like to show that $\mathbb{X} = \mathbb{B}(\mathcal{L})$.    Suppose for contradiction that there exists some $A \in  \mathbb{B}(\mathcal{L})\setminus \mathbb{X}$.  We know that $\mathbb{X}$ is nonempty as $\mathcal{K}^+\cap \mathcal{L}^+=\mathcal{L}^+ \in \mathbb{X}$, hence there exists a containment minimal element $B \in \mathbb{X}$ such that $A \subsetneq B$.   
 By the supersolvability of $\mathbb{B}(\mathcal{L})$ applied to the pair of building sets $A$ and $B$, we know that for any linear extension $<$ of $\mathcal{L}$ and $x =\text{min}_{<}(B\setminus A)$ that $B\setminus \{x\} \in \mathbb{B}(\mathcal{L})$.  
    By our minimality assumption on $B$, we have $\sigma(B\setminus \{x\})\cap\mathcal{L}^+\neq B\setminus \{x\}$.  So there exists some element $y \in \sigma(B\setminus \{x\})\cap\mathcal{L}^+$ and $y \notin B\setminus \{x\}$.  By condition (\ref{closure2}) of Definition \ref{closure}, we have that $B\setminus \{x\} \subsetneq \sigma(B\setminus \{x\})\cap\mathcal{L}^+ \subseteq \sigma(B)\cap\mathcal{L}^+ = B$. 
 Thus $\sigma(B\setminus \{x\})\cap\mathcal{L}^+=B$ and $y  = x$.  

Let $\{y_1, \ldots, y_t\} = \text{max} (B\setminus \{x\})_{\leq x}$.  In what follows, we will employ the description of $\sigma(B\setminus \{x\})$ given in Proposition \ref{closurealgo}.  Let $X_0= (B\setminus \{x\}) \cup I(\mathcal{K})$.  
Because $\hookrightarrow$ is consistent, and $I(\mathcal{L})\subseteq B\setminus \{x\}$ as $B\setminus \{x\}\in \mathbb{B}(\mathcal{L})$, we know that $(X_0)_{\leq x}= (B\setminus \{x\})_{\leq x}$.  We utilize induction for establishing that for every $j \in \mathbb{Z}_{\geq 0}$, the collection $\{ (X_j)_{\leq y_r}\setminus \{\hat{0}\}:1\leq r\leq t\}$ forms a partition of $(X_j)_{\leq x}\setminus \{\hat{0}\}$, i.e. the following two statements hold:

\begin{enumerate}
    \item\label{firstinductivecondition} For every $z \in (X_j)_{\leq x}$, there exists some $y_r$ such that $z\leq y_r$. 
    \item\label{secondinductivecondition} If $y_r\neq y_s$ then $(X_j)_{\leq y_r} \cap (X_j)_{\leq y_s}= \{\hat{0}\}$.
\end{enumerate}  

Statement (\ref{firstinductivecondition}) for $(X_0)_{\leq x}$ follows by the definition of the $\{y_r\}$.  Statement (\ref{secondinductivecondition}) for $(X_0)_{\leq x}$ is a consequence of the fact that $ B\setminus \{x\}$ is a building set for $\mathcal{L}$.
Now assume that the statements (\ref{firstinductivecondition})  and (\ref{secondinductivecondition})  hold for some $(X_i)_{\leq x}$. Let $z \in (Y_i)_{\leq x}$ and $a,b \in (X_i)_{\leq x}$ such that $a\wedge b \neq \hat{0}$ with $a\vee_{\mathcal{K}} b =z$.  By statement (\ref{secondinductivecondition}) for $(X_i)_{\leq x}$, it must be that there exists some $r$ such that $a,b \leq y_r$, hence $a\vee_{\mathcal{K}}b \leq y_r$ implying statement (\ref{firstinductivecondition}) for $(X_{i+1})_{\leq x}$.  Furthermore, it is clear that for any $s\neq r$, we have that $a\vee b \nleq y_s$, otherwise we would have that $a,b \leq y_s$, contradicting statement (\ref{secondinductivecondition}) for $(X_i)_{\leq x}$.   Thus statement (\ref{secondinductivecondition}) holds for $(X_{i+1})_{\leq x}$.   Taking $j$ large enough so that $X_j = \sigma(B\setminus \{x\})$, we find that for any $z \in \sigma(B\setminus \{x\})_{\leq x}$, there exists some $r$ such that $z\leq y_r$, but this contradicts the above assumption that $x \in \sigma(B\setminus \{x\})_{\leq x}$.  This concludes the proof that $\mathbb{X} = \mathbb{B}(\mathcal{L})$.

The argument above establishes that $\mathbb{B}(\mathcal{L}) \subseteq \{B\cap 
\mathcal{L}^+:B\in \mathbb{B}(\mathcal{K}), I(\mathcal{L})\subseteq B\}$, thus to complete the proof of (\ref{1stpart}) from Theorem \ref{1strest}, we will demonstrate the opposite inclusion.   Let $\mathbb{Y} = \{ B \in \mathbb{B}(\mathcal{K}):I(\mathcal{L})\subseteq B \, , \, B\cap\mathcal{L}^+ \in \mathbb{B}(\mathcal{L})\}.$ We will show that $\mathbb{Y} =  \{B \in \mathbb{B}(\mathcal{K}):I(\mathcal{L})\subseteq B\}$.  Suppose for contradiction that $\mathbb{Y} \neq  \{B \in \mathbb{B}(\mathcal{K}):I(\mathcal{L})\subseteq B\}$ and let $A \in\{B \in \mathbb{B}(\mathcal{K}):I(\mathcal{L})\subseteq B\} \setminus \mathbb{Y}$.  We know that $\mathcal{K}^+
\in \mathbb{Y} $, so $\mathbb{Y}$ is nonempty, hence there exists a containment minimal building set $B \in \mathbb{Y}$ such that $A \subsetneq B$.    By the supersolvability of $\mathbb{B}(\mathcal{K})$ applied to the pair of building sets $A$ and $B$, we know that for any linear extension $<$ of $\mathcal{K}$ and $x =\text{min}_{<}(B\setminus A)$, we have $B\setminus \{x\} \in \mathbb{B}(\mathcal{K})$.  By our minimality assumption on $B$, we have that $(B\setminus \{x\})\cap \mathcal{L}^+ \notin \mathbb{B}(\mathcal{L})$.  We know that $I(\mathcal{L})\subseteq A$, hence $x \notin I(\mathcal{L})$ and $I(\mathcal{L}) \subseteq (B\setminus \{x\})\cap \mathcal{L}^+$.  
 By Proposition \ref{extremebuild}, there must exist some $y,z \in (B\setminus \{x\})\cap \mathcal{L}^+$ such that $y\wedge z \neq \hat{0}$ and $y\vee^{\mathcal{L}} z = x$. 
We know that $B\setminus \{x\} \in \mathbb{B}(\mathcal{K})$.  Take $\{w_1, \ldots,w_k\} = \text{max}\, (B\setminus \{x\})_{\leq x}$, and recall that for each $1\leq i<j \leq k$, we have that $[\hat{0},w_i]\cap[\hat{0},w_j] = \hat{0}$.  Note that since $x \notin B\setminus \{x\}$, we must have $k\geq 2$.  
Also, if $u \in (B\setminus \{x\})_{\leq x}$ and $u \wedge w_i \neq \hat{0}$, then $u\leq w_i$.  As $y,z \in B\setminus \{x\}$ and $\hat{0} \neq y\wedge z$, we know that $y\vee^{\mathcal{K}} z \in B\setminus \{x\}$.  We may assume without loss of generality that $y\vee^{\mathcal{K}} z \leq w_1$.  Let $\{x_1, \ldots, x_t\} = \text{max}\, I(\mathcal{L})_{\leq x}\subseteq B\setminus \{x\}$. 
Suppose that for some $s$ with $1\leq s\leq t$, we have $x_s \wedge y = x_s \wedge z = \hat{0}$.  Because $I(\mathcal{L}) \in \mathbb{B}(\mathcal{L})$, we know that for $1\leq i\leq t$ there exist $y_i,z_i \leq x_i$ such that $\vee_{1\leq i \leq t}^{\mathcal{L}} y_i = y$ and $\vee_{1\leq i \leq t}^{\mathcal{L}} z_i = z$.  By assumption, $y_s = z_s = \hat{0}$, so $y \vee^{\mathcal{L}} z \leq \vee_{i \neq s}^{\mathcal{L}} x_i < \vee_{1\leq i \leq t}^{\mathcal{L}} x_i  = x$, a contradiction.
Thus for each $i$, the intersection $x_i \wedge y \neq \hat{0}$ or  $x_i \wedge z \neq \hat{0}$, hence $x_i \wedge w_1 \neq \hat{0}$.  As $x_i \in (B\setminus \{x\})_{\leq x}$, we must have $x_i \leq w_1$, therefore $I(\mathcal{L})_{\leq x} \subseteq [\hat{0},w_1]$.  Because $\mathcal{L} \hookrightarrow \mathcal{K}$ is consistent, $I(\mathcal{K})_{\leq x} \subseteq I(\mathcal{L})_{\leq x}$.  We conclude that $I(\mathcal{K})_{\leq x}\subseteq [\hat{0}, w_1]$, implying $k=1$, a contradiction.

   (\ref{2ndpart}):  
Suppose for contradiction that $B \in  \mathbb{B}(\mathcal{K})$, and $I(\mathcal{L}) \subseteq B$, but $\mathcal{N}_{\mathcal{L}}(B \cap \mathcal{L}^+)\nsubseteq \mathcal{N}_{\mathcal{K}}(B).$  Let $N \in \mathcal{N}_{\mathcal{L}}(B \cap \mathcal{L}^+)$ such that $N \notin \mathcal{N}_{\mathcal{K}}(B)$.  There exists some $x_1, \ldots, x_k \in N$ which are incomparable, such that $y=\bigvee_{1\leq i\leq k}^{\mathcal{K}} x_i \in B$ but $x=\bigvee_{1\leq i\leq k}^{\mathcal{L}} x_i \in \mathcal{L}^+ \setminus B$. Let $\{y_1, \ldots, y_l \}= 
\text{max}\,B_{\leq x}$.  Then $l\geq 2$ and the intervals $[\hat{0},y_i]$ intersect at $\hat{0}$.  By Lemma \ref{wedgeineqlem}, $y \leq x$, thus there exists exactly one $i$ such that $y \leq y_i$.  
But then by Lemma \ref{nestedlem} applied to $x_1, \ldots, x_k$, we know that for any $x' \in (B \cap \mathcal{L}^+)_{\leq x}$,  there exists some $x_t$ such that $x' \leq x_t$, hence $x' \leq y \leq y_i$.  It follows that for any $y_j$ with $j \neq i$, we have that $(B\cap \mathcal{L}^+ )_{\leq y_j} = \emptyset$.  On the other hand, there must exist some $y' \in I(\mathcal{K})_{\leq y_j}$. 
  Clearly, $I(\mathcal{K})_{\leq y_j} \subseteq I(\mathcal{K})_{\leq x}$, and by our assumption that $\mathcal{L} \hookrightarrow \mathcal{K}$ is consistent, we must have that $y' \in I(\mathcal{L})_{\leq x} \subseteq (B\cap \mathcal{L}^+ )_{\leq y_i}$, a contradiction.

\end{proof}

\section{Acknowledgements}

 We thank Nate Bottman, Federico Castillo, Graham Denham, Alexei Oblomkov, Arnau Padrol, Vincent Pilaud, Dasha Poliakova, and Raman Sanyal for helpful discussions about building sets, and we thank Bryan Gillespie for carefully explaining his work during the 2017 MSRI Semester on ``Geometric and Topological Combinatorics''.  Some of our discussions took place at the Simons Center for Geometry and Physics Workshop on ``Combinatorics and Geometry of Convex Polyhedra" in March 2023.  S.B. thanks the MSRI (now SLMath) and the Simons Center for their hospitality and excellent working conditions.  We thank Christin Bibby, Graham Denham, Matt Larson, Arnau Padrol, Vincent Pilaud, Vic Reiner, and Raman Sanyal for their helpful feedback on an earlier version of this article.  We also thank an anonymous referee for feedback that helped improve this article.

\bibliographystyle{abbrv}
\bibliography{references.bib}
\end{document}